\documentclass[11pt]{elsarticle}

\usepackage{amsfonts}
\usepackage{amsmath}
\usepackage{amssymb}
\usepackage{amsthm}

\usepackage{graphicx}
\usepackage{float}
\restylefloat{figure}

\journal{Computers \& Mathematics with Applications}

\theoremstyle{remark}

\def\RR{\hbox{I\kern-.2em\hbox{R}}}
\numberwithin{equation}{section}
\restylefloat{figure}

\usepackage{graphicx} %%For loading graphic files
\usepackage{amsmath}
\usepackage{amsthm}
\usepackage{amsfonts}
\usepackage{enumerate}
\usepackage{enumitem}

\newtheorem{Lemma}{Lemma}
\newtheorem{Th}{Theorem}

\newtheorem{Rem}{Remark}
\newtheorem{Ex}{Example}
\usepackage{cases}
\usepackage{enumerate}
\usepackage{fullpage} %%Smaller margins = more text per page.

\begin{document}  

\begin{frontmatter}

\title{Competitive-cooperative models with various diffusion strategies}
%% Other Packages %%%%%%%%%%%%%%%%%%%%%%%%%%%%%%%%%%%%%%%%%%%

\author[eb]{E.~Braverman
%\thanks{Corresponding author, e-mail maelena@ucalgary.ca, Phone 1-403-220-3956, Fax 1-282-5150}
}
\ead{maelena@math.ucalgary.ca}
\address[eb,mk]{Department of Mathematics and Statistics, University of Calgary,
2500 University Drive N. W.,\\ Calgary, AB T2N 1N4, Canada}
\author[mk]{Md. Kamrujjaman}
%\address[mk]{Department of Mathematics and Statistics, University of Calgary,
%2500 University Drive N. W.,\\ Calgary, AB T2N 1N4, Canada}

\begin{abstract}
The paper is concerned with different types of dispersal chosen by competing 
species. We introduce a model with the diffusion-type term  
$\nabla \cdot \left[ a \nabla \left( u/P \right) \right]$ which includes some 
previously studied systems as special cases, where a positive space-dependent function $P$ can be interpreted as a 
chosen dispersal strategy. The well-known result that if 
the first species chooses $P$ proportional to the carrying capacity while
the second does not then the first species will bring the second one to extinction,
is also valid for this type of dispersal. However, we focus on the case when 
the ideal free distribution is attained as a combination of the two strategies adopted by the two species.
Then there is a globally stable coexistence equilibrium, its uniqueness is justified.  
If both species choose the same dispersal strategy, non-proportional to the carrying capacity,
then the influence of higher diffusion rates is negative, while of higher intrinsic growth rates is 
positive for survival in a competition. This extends the result of  
[J. Math. Biol. {\bf 37} (1) (1998), 61--83] for the regular diffusion 
to a more general type of dispersal.

\end{abstract}
\end{frontmatter}

%{\bf Keywords:}
%ideal free distribution; ideal free pair;
%dispersal strategy; global attractivity; system of partial differential equations; coexistence; 

%{\bf AMS  subject classification:} 92D25, 35K57 (primary), 35K50, 37N25

\section{Introduction}

The study of population dynamics started with the description of the total population size and its dynamics.
However, taking into account the spatial structure is essential for understanding of invasion, survival and 
extinction of species. In order to involve movements in spatially distributed systems, the dispersal strategy
should be specified. The ideas of spreading over the domain to avoid overcrowding combined with advection towards 
higher available resources were incorporated in the model in \cite{Averill,C4,C2,C3,C5}.
Dispersal design in \cite{C2} was based on the notion of the ideal free distribution, i.e. such distribution that 
any movement in an ideally distributed system will decrease the fitness of moving individuals.  
An ideal free distribution in a temporally constant but spatially heterogeneous environment
is expected to be a solution of the system. There were several possible approaches to model dispersal
in such a way that the ideal free distribution is a stationary solution of the equation, see, for example,
\cite{Brav,C2,C3,C5}.
 
Most of the previous research \cite{C2,C3,C5,Dock,Kor2,Kor4} was focused on evolutionarily stable strategies
which provide advantages in a competition.
This allowed to answer the question: what parameters or strategies should a spatially distributed population choose 
so that its habitat cannot be invaded by a species choosing alternative strategy and having other parameters?
The effect of environment heterogeneity was studied in \cite{He}. 
The case when a preferred diffusion strategy alleviated a negative effect of less efficient resources exploitation
leading to possible coexistence was recently studied in \cite{Brav1}.
However, the question how spatial distribution can promote coexistence compared to homogeneous environment got much less attention. 
Coexistence of various species relying on the same resources is quite common in nature. 
Survival of a certain population can be secured if it finds a certain niche in the resource consumption.
Adaptation of different parts of the habitat (for example, shallow waters and deep waters) by two (or more)
competing for resources populations can lead to coexistence.  

The purpose of this paper is two-fold.
\begin{enumerate}
\item
The first goal is to unify different approaches to diffusion strategies and dispersal, for example,
the advection to better environments in \cite{C4} and the model where not the population density
but its ratio to locally available resources diffuses \cite{Brav}. 
A type of dispersal which includes \cite{Brav,C4} is introduced in Section~\ref{sec:dispersal}.
\item
The second goal is to consider the case when dispersal strategies guarantee coexistence.
If the carrying capacity of the environment is a combination of these strategies, the ideal free distribution
is attained at the coexistence stationary solution where the first population can be more abundant
in some areas of the habitat, while the second one can be dense in other areas. Altogether, the two populations 
use the resources in an optimal way, and the sum of two densities coincides with the carrying capacity
at any place of the spatial domain. This coexistence equilibrium is unique and globally attractive, see
Section~\ref{sec:directed}.
\end{enumerate}

In addition, we extend the results of \cite{Dock} to the case when the diffusion is not necessarily regular but
the diffusion strategy is the same for both species, and it is different from \cite{Brav} where, as the diffusion
coefficient tends to infinity, the solution tends to the carrying capacity which coincides with the ideal free distribution.
If this is the case, higher diffusion coefficient leads to a disadvantage in a competition. However,
Section~\ref{sec:influence} also gives an alternative interpretation to this result: higher intrinsic growth rates
(assuming the two have the same dispersal strategy) give an evolutionary advantage.
Finally, Section~\ref{sec:discussion} includes discussion of the results obtained and states some open questions.

\section{Dispersal Modeling}
\label{sec:dispersal}

When choosing the type of diffusion, we focus on the ideal free distribution which is expected to be a solution of
the equation, in the absence of other species. A related idea is that there is a movement towards higher available 
resources, not just lower densities, as the regular diffusion $\Delta u$ would suggest.
For our modeling, we assume the null hypothesis that not $u$ but $u/P$ is subject to advection or diffusion, where 
$P$ is a diffusion strategy chosen by a species whose density at time $t$ and point $x$ is $u(t,x)$. Here $P$ is a positive 
and smooth in the domain $\Omega$ function.
Assuming space-dependent rate of advection $a(x)$, we consider
\begin{equation}
\label{distrib}
\nabla \cdot \left[ a(x) \nabla \left( \frac{u(t,x)}{P(x)} \right) \right]
= \nabla \cdot \left[ \frac{a(x)}{P(x)} \left( \nabla u(t,x)- u(t,x) \frac{\nabla P}{P} \right) \right].
\end{equation}
The space-distributed $P$ is treated as a diffusion strategy in the following sense: if $a(x) \to +\infty$,
the density $u(t,x)$ would be proportional to $P$, whatever growth law we choose. 
Accepting \eqref{distrib} as a dispersal model, we consider the following system of two competing species obeying 
the logistic growth rule, with the Neumann boundary conditions
\begin{align}\label{ssystem}
\begin{cases}
\frac{\displaystyle \partial u}{\displaystyle \partial t}  
= \nabla \cdot \left[ \frac{\displaystyle a_1(x)}{\displaystyle P(x)} \left( \nabla u(t,x)- u(t,x) \frac{\displaystyle 
\nabla P}{\displaystyle P} \right) \right]
+r(x)u(t,x)\left(1-\frac{\displaystyle u(t,x)+ v(t,x)}{\displaystyle K(x)}\right), \vspace{3mm}
\\
\frac{\displaystyle \partial v}{\displaystyle \partial t} 
= \nabla \cdot \left[ \frac{\displaystyle a_2(x)}{\displaystyle Q(x)} \left( \nabla v(t,x)- v(t,x) \frac{\nabla Q}{Q} 
\right) \right] + r(x)v(t,x)\left(1-\frac{\displaystyle  u(t,x)+v(t,x) }
{\displaystyle K(x)}\right),\\
t>0,\;x\in{\Omega},\\
\frac{\displaystyle \partial u}{\displaystyle \partial n} - \frac{\displaystyle u}{\displaystyle P} \frac{\displaystyle 
\partial P}{\displaystyle \partial n}  
=\frac{\displaystyle \partial v}{\displaystyle \partial n}  - 
\frac{\displaystyle v}{\displaystyle Q} \frac{\displaystyle \partial Q}{\displaystyle \partial n} =0,~x\in 
{\partial}{\Omega},\\
u(0,x)=u_0(x), \;v(0,x)=v_0(x),\;x\in{\Omega}.
\end{cases}
\end{align}
Let us note that, for smooth positive $P$ and $Q$, the boundary conditions in \eqref{ssystem}
are equivalent to 
\begin{equation}
\label{boundary}
\frac{\displaystyle \partial }{\displaystyle \partial n} \left( \frac{\displaystyle u}{\displaystyle P} \right)=0, ~
\frac{\displaystyle \partial }{\displaystyle \partial n} \left( \frac{\displaystyle v}{\displaystyle Q} \right)=0, ~
~x\in {\partial}{\Omega}.
\end{equation}

The three most common particular cases are outlined below:
\begin{enumerate}
\item
If either both $P$ and $a_1$ or both $Q$ and $a_2$ are constant, then either the first or the second equation
incorporates a regular diffusion term $d_1 \Delta u$ or $d_2 \Delta v$, where $d_1=a_1/P$ or $d_2=a_2/Q$. 
\item
If  $a_1$ is space-independent, in the first equation of \eqref{ssystem} we obtain the type of dispersal 
$\Delta(u/P)$. In the particular case when $P \equiv K$ (or $P$ is proportional to $K$) we have the term 
$\Delta(u/K)$, first introduced in \cite{Brav} and later considered in \cite{Kor1}-\cite{Kor3}.
%\cite{Kor1,Kor2,Kor4,Kor3}. 
If $Q$ is constant, while $a_2$ is proportional to $1/K$, we have the dispersal type $\nabla \cdot (\frac{1}{K}\nabla v)$ 
which was considered in \cite{Kor2,Kor4}.
\item
If $a_1 =\mu_1 P$, $\ln P =\mu_2 K$, 
where $\mu_i$, $i=1,2$ are space-independent, 
and $r=K$, we 
obtain the directed advection model of the type
\begin{equation}
\label{CC_equation}
\frac{\partial u}{\partial t} = 
\nabla \cdot \left[ \mu \nabla u - \alpha u \nabla K \right] + u(K-u),
\end{equation}
where $\mu=\mu_1$, $\alpha=\mu_1 \mu_2$. Equation \eqref{CC_equation} was
considered in \cite{Averill,C4,C2,C3,C5}, see also references therein.
\end{enumerate}

Thus, \eqref{ssystem} generalizes most of earlier considered dispersal strategies. However, it should be mentioned
that, for constant $a_1/P$, there are publications where $P$ is not assumed to be positive everywhere on $\Omega$.

\section{Directed Diffusion Competition Model}
\label{sec:directed}

%Next, let us introduce conditions on the parameters of \eqref{ssystem} and consider possible equilibrium solutions.
We assume that the domain $\Omega$ is an open bounded region 
in $\mathbb{R}^n$ with ${\partial}{\Omega}\in C^{2+\beta}$, 
$\beta>0$, the functions
$r(x)$, $a_i(x)$, $i=1,2$, $P(x)$, $Q(x)$ and $K(x)$ are continuous and positive on $\overline{\Omega}$;
moreover, $a, P, Q \in C^2 (\Omega)$.

Next, let us proceed to the study of stationary solutions (positive equilibria) of \eqref{ssystem}.
Problem \eqref{ssystem} is a monotone dynamical system \cite{CC,Pao,Hs}. If all equilibrium solutions but one are unstable, we would 
be able to conclude that the remaining equilibrium is globally asymptotically stable. We start with the trivial 
equilibrium, and, similarly to \cite{Kor2,Kor4}, verify the following result.

\begin{Lemma}
\label{zero_equilibrium}
The zero solution of \eqref{ssystem} is unstable; moreover, it is a repeller.
\end{Lemma}

Let functions $u^{\ast}$ and $v^{\ast}$ be solutions of the single-species stationary models
corresponding  to the first and the second equations in  \eqref{ssystem}
\begin{equation}\label{eq_u}
\nabla \cdot \left[ a_1(x)  \nabla \left( \frac{\displaystyle u^{\ast}(x)}{\displaystyle P(x)}\right) \right]
+r(x)u^{\ast}(x)\left(1-\frac{\displaystyle u^{\ast}(x)}{\displaystyle K(x)}\right) = 0,\; x\in\Omega,\;
\frac{\displaystyle  \partial (u^{\ast} /P)}{\displaystyle \partial n}=0,\; x\in\partial\Omega, 
\end{equation}
\begin{equation}\label{eq_v}
\nabla \cdot \left[ a_2(x) \nabla \left( \frac{\displaystyle v^{\ast}(x)}{\displaystyle Q(x)}\right) \right]
+r(x)v^{\ast}(x)\left(1-\frac{\displaystyle v^{\ast}(x)}{\displaystyle K(x)}\right) = 0,\; x\in\Omega,\;
\frac{\displaystyle  \partial (v^{\ast} /Q)}{\displaystyle \partial n}=0,\; x\in\partial\Omega ,
\end{equation}
respectively.

In future, we will need the following two auxiliary statements. 

\begin{Lemma}\label{Lsteady2}
Let $u^{\ast}$ be a positive solution of (\ref{eq_u}), then
\begin{equation}\label{eq_u2_1}
\int \limits_\Omega r(x)P(x)\left(\frac{\displaystyle u^*(x)}{\displaystyle K(x)}-1\right)\,dx=
\int \limits_\Omega \frac{a_1(x)|\nabla (u^*/P)|^2}{(u^*/P)^2}\,dx.
 \end{equation}
If $P(x)$ and $K(x)$ are linearly independent on $\Omega$, then
\begin{equation}\label{eq_u2}
\int \limits_\Omega r(x)P(x)\left(\frac{\displaystyle u^*(x)}{\displaystyle K(x)}-1\right)\,dx > 0. 
\end{equation}
\end{Lemma}
\begin{proof}
Since $u^{\ast}>0$ and $P(x)>0$ for any $x\in \Omega$, dividing equation (\ref{eq_u}) by $u^{\ast}/P$, we 
obtain
\begin{equation}\label{eq_u3}
\frac{\nabla \cdot [a_1 \nabla (u^{\ast}/ P)]}{(u^{\ast}/ P)}+r(x)P(x)\left(1-\frac{\displaystyle 
u^{\ast}(x)}{\displaystyle 
K(x)}\right) = 0,\; x\in\Omega,\;
\frac{\displaystyle  \partial (u^{\ast} /P)}{\displaystyle \partial n}=0,\; x\in\partial\Omega 
\end{equation} 
Integrating (\ref{eq_u3}) over the domain $\Omega$ using boundary conditions in (\ref{eq_u3}), we have
\begin{equation}\label{eq_u4}
\int \limits_\Omega \frac{a_1 |\nabla (u^{\ast}/P)|^2}{(u^{\ast}/P)^2}\,dx+\int \limits_\Omega r(x)P(x)
\left(1-\frac{\displaystyle u^{\ast}(x)}{\displaystyle K(x)}\right)\,dx=0
\end{equation}
Therefore
\begin{equation}\label{eq_u5}
\int \limits_\Omega r(x)P(x)\left(\frac{\displaystyle u^{\ast}(x)}{\displaystyle K(x)}-1\right)\,dx=
\int \limits_\Omega \frac{a_1(x)|\nabla (u^{\ast}/P)|^2}{(u^{\ast}/P)^2}\,dx> 0,
  \end{equation}
unless $u^{\ast}/P$ is identically equal to a positive constant. However, substituting
$u^{\ast}(x)=c P(x)$ into \eqref{eq_u}, we obtain $cP(x) \equiv K(x)$, $x \in \Omega$, which contradicts to our 
assumption that $P$ and $K$ are linearly independent on $\Omega$.
\end{proof}

A similar result is valid for $v^{\ast}$ whenever $Q$ and $K$ are linearly independent.

The analogues of the following statements for less general diffusion types were obtained in \cite{Kor2,Kor4},
for completeness we present the proof here.

\begin{Lemma}
\label{Lsteady1}
Suppose that $u^{\ast}$ is a positive solution of \eqref{eq_u}, while
$P(x)$ and $K(x)$ satisfy 
$\nabla \cdot [a_1(x)\nabla (K(x)/P(x))] \not\equiv 0$ on $\Omega$. 
Then
\begin{equation}\label{eq_est1abc}
\int\limits_{\Omega}r(x) K(x) \,dx>\int\limits_{\Omega} r(x) u^{\ast}(x)\,dx.
\end{equation}
\end{Lemma}
\begin{proof}
Substituting $u=u^{\ast}$ and $v \equiv 0$ in the first equation of \eqref{ssystem} and integrating 
over $\Omega$, using the boundary conditions gives
$$0=\int_{\Omega} r u^{\ast} \left(1-\frac{\displaystyle u^{\ast}}{\displaystyle K}\right)~dx=
- \int_{\Omega} r K \left(1-\frac{\displaystyle u^{\ast}}{\displaystyle K}\right)^2 dx
+\int_{\Omega} r K \left(1-\frac{\displaystyle u^{\ast}}{\displaystyle K}\right)~dx.
$$
Since the first integral in the right-hand side is non-positive, the second integral is non-negative. Moreover, 
it is positive unless $u^{\ast} \equiv K$ which would imply 
$\nabla \cdot [a_1(x)\nabla (K(x)/P(x))] \equiv 0$ on $\Omega$. 
The contradiction justifies inequality \eqref{eq_est1abc}. 
\end{proof}

Also, if 
$\nabla \cdot [a_2(x)\nabla (K(x)/Q(x))] \not\equiv 0$ on $\Omega$, we have 
 \begin{equation}\label{eq_v1}
 \int \limits_\Omega r(x) K(x)\left(1-\frac{\displaystyle v^{\ast}(x)}{\displaystyle 
K(x)}\right)\,dx>0.
 \end{equation}

Similarly to Lemma~\ref{Lsteady1}, the following result is justified.

\begin{Lemma}
\label{Lsteady1abc}
Suppose that $(u_s,v_s)$ is a positive stationary solution of \eqref{ssystem}, such that
$u_s(x)+v_s(x) \not\equiv K(x)$.
Then
\begin{equation}\label{eq_est1}
\int\limits_{\Omega}r(x) K(x)\left( 1 -\frac{u_s(x)+v_s(x)}{K(x)} \right) \,dx> 0.
\end{equation}
\end{Lemma}

\begin{Lemma}\label{semi_p1}
Suppose that $P$, $K$ and $Q$, $K$ are two pairs of linearly independent on $\Omega$ functions,
while $K(x)\equiv \alpha P+\beta Q$ for some $\alpha>0$, $\beta>0$.
Then the semi-trivial steady state $(u^{\ast}(x),0)$ of (\ref{ssystem}) is unstable.
\end{Lemma}
\begin{proof} 
Consider the eigenvalue problem  associated with the second equation in (\ref{ssystem})
around the equilibrium $(u^{\ast}(x),0)$ 
 \begin{equation}\label{eig_p1}
\nabla \cdot \left[ a_2(x) \nabla \left(\frac{\displaystyle \psi(x) }{\displaystyle Q(x)}\right) \right]
+r(x)\psi(x)  \left(1-\frac{\displaystyle u^{\ast}(x)}{\displaystyle K(x)}\right)=\sigma \psi(x),\; x\in \Omega,\;
 \frac{\displaystyle \partial (\psi/Q)}{\displaystyle \partial n}=0,\; x\in\partial\Omega
\end{equation}
The principal eigenvalue of (\ref{eig_p1}) is defined as \cite{CC}
$$
\sigma_1 =
\sup_{\psi \neq 0, \psi\in W^{1,2}} \left. \left[
-\int \limits_\Omega a_2 |\nabla (\psi/Q)|^2\,dx
+\int \limits_\Omega r(x) \frac{\psi^2}{Q} \left(1-\frac{\displaystyle u^{\ast}}{\displaystyle K}\right)\,dx\right]
\right/
\int \limits_\Omega \frac{\psi^2}{Q}\,dx.
$$
 
Choosing $\psi(x)=\sqrt{\beta} Q(x)$ and denoting
$\displaystyle
M:=\int \limits_\Omega \beta Q(x) \,dx ,
$
we observe that the principal eigenvalue is not less than
\begin{align}\label{peig_p1}
\sigma_1 &\geq \frac{1}{M} \int \limits_\Omega r(x)\beta Q(x)\left(1-\frac{u^{\ast}(x)}{K(x)}\right) 
\,dx\nonumber\\
&=  \frac{1}{M} \int \limits_\Omega r(x)(K(x)-\alpha P(x))\left(1-\frac{u^{\ast}(x)}{K(x)}\right) 
\,dx\nonumber\\
&= \frac{1}{M} \int \limits_\Omega r(x)K(x)\left(1-\frac{u^{\ast}(x)}{K(x)}\right) \,dx
+\frac{\alpha}{M} \int \limits_\Omega r(x) P(x)\left(\frac{u^{\ast}(x)}{K(x)}-1\right) \,dx >0,\nonumber
\end{align}
since the first term is non-negative by Lemma~\ref{Lsteady1}, while the second is positive by 
Lemma~\ref{Lsteady2}.
Therefore, $\sigma_1$ is positive, and 
the semi-trivial steady state $(u^{\ast}(x),0)$ of (\ref{ssystem}) is unstable.
\end{proof}

Similarly, we obtain that, under the assumptions of Lemma~\ref{semi_p1}, $(0,v^{\ast}(x))$ is also unstable.

\begin{Lemma}\label{semi_p1abc}
Suppose that $P\equiv K$, while $K$ and $Q$ are linearly independent on $\Omega$.
Then the semi-trivial steady state $(0,v^{\ast}(x))$ of \eqref{ssystem} is unstable.
\end{Lemma}

\begin{Lemma}\label{coex_r1}
Assume that $P(x)$ and $Q(x)$ are linearly independent on $\Omega$, and $K(x)\equiv \alpha P+\beta Q$, with $\alpha>0$, $\beta>0$.
Then the system (\ref{ssystem}) has a unique positive coexistence equilibrium  $(u_s,v_s)\equiv (\alpha P(x),\beta Q(x))$.
\end{Lemma}
\begin{proof}
A stationary solution $(u_s, v_s)$ of system (\ref{ssystem}) satisfies 
\begin{equation}\label{coex_r2}
\begin{cases}
\nabla \cdot  \left[ a_1(x) \nabla \left(\frac{\displaystyle u_s(x)}{\displaystyle P(x)}\right) \right]+
r(x)u_s(x)\left(1-\frac{\displaystyle u_s(x)+v_s(x)}{\displaystyle K(x)}\right)=0,\;x\in\Omega, \vspace{2mm} \\
\nabla \cdot \left[ a_2(x) \nabla \left(\frac{\displaystyle v_s(x)}{\displaystyle Q(x)}\right) \right] +
r(x)v_s(x)\left(1-\frac{\displaystyle u_s(x)+v_s(x) }{\displaystyle K(x)}\right) =0,\;x\in\Omega, \vspace{2mm} \\
\frac{\displaystyle\partial (u_s/P)}{\displaystyle\partial n}
=\frac{\displaystyle\partial (v_s/Q)}{\displaystyle\partial n}=0,\;x\in\partial\Omega.
\end{cases}
\end{equation}
The direct substitution, due to $K(x)\equiv \alpha P+\beta Q$, 
immediately implies that $(\alpha P(x),\beta Q(x))$ is a coexistence stationary solution of (\ref{coex_r2}).
To show the uniqueness, assume that $(u_s, v_s) \not\equiv (\alpha P(x),\beta Q(x))$ is a coexistence equilibrium 
satisfying (\ref{coex_r2}).

Adding the first two equations of (\ref{coex_r2}), integrating  over $\Omega$ and taking into account the 
Neumann boundary conditions, we have
\begin{equation}
\label{s1_ceq2}
\int\limits_\Omega r(u_s+v_s)
\left(1-\frac{\displaystyle u_s+v_s}{\displaystyle K}\right)\,dx= 0,
\end{equation}
which implies
\begin{equation}\label{c5_e1}
\int \limits_\Omega rK\left(1-\frac{\displaystyle u_s+v_s}{\displaystyle K}\right)\,dx=
\int \limits_\Omega r K\left(1-\frac{\displaystyle u_s+v_s}{\displaystyle K}\right)^2\,dx> 0,
\end{equation}
unless $u_s+v_s\equiv K$. However, if $u_s+v_s\equiv K$, the function $w_s=u_s/P$ should satisfy 
$$ \nabla \cdot (a_1(x) \nabla w_s)=0, ~x \in \Omega, ~~ \partial w_s/\partial n=0, ~ x \in \partial \Omega,$$ 
and thus $w_s$ is constant by the Maximum Principle \cite[Theorem 3.6]{Gilb},
which means that $u_s/P$ is constant on $\Omega$.
Similarly, $v_s/Q$ is constant on $\Omega$. From the fact that $K=\alpha P +\beta Q$ is
the only possible representation of $K$ as a linear combination of $P,Q$ (otherwise,
$P$ and $Q$ are linearly dependent), we obtain $u_s=\alpha P$, $v_s=\beta Q$.

Further, dividing the second equation of \eqref{coex_r2} by $v_s/Q$ and integrating over $\Omega$, we obtain
\begin{equation}\label{c5_e2}
\int \limits_\Omega r Q \left(\frac{\displaystyle u_s+v_s}{\displaystyle K}-1\right)\,dx=
\int \limits_\Omega a_2 \frac{|\nabla (v_s/Q)|^2}{(v_s/Q)^2}\,dx\geq 0.
\end{equation}

Next, let $u_s+v_s \not\equiv K$.
Consider the eigenvalue problem
\begin{equation}
%\begin{array}{ll}
\nabla \cdot \left[ a_1 \nabla \left(\frac{\phi}{P} \right)\right]
+r\phi \left(1-\frac{\displaystyle u_s+v_s}{\displaystyle K}\right)
=\sigma\phi, \;x\in\Omega,
\frac{\displaystyle \partial (\phi/P)}{\displaystyle \partial n}=0, \;x\in\partial\Omega
%\end{array} 
\label{c5_eig1}
\end{equation}
Its principal eigenvalue $ \sigma_1 $ is given by  \cite{CC}
$$
\sigma_1 
= \left. \sup_{\phi \neq 0, \phi\in W^{1,2}} \left[-
\int \limits_\Omega a_1 |\nabla (\phi/P)|^2\,dx
+\int \limits_\Omega r(x)\frac{\phi^2}{P} \left(1-\frac{\displaystyle u_s+v_s}{\displaystyle K}\right)\,dx\right] \right/
\int \limits_\Omega \frac{\phi^2}{P}\,dx
$$
Consider
%$
%\sigma_1\int\limits_{\Omega}\alpha P(x) \,dx \geq \int\limits_{\Omega}r(x) \alpha P(x) \left(1-\frac{\displaystyle u_s+v_s}
%{\displaystyle K}\right)\,dx
%$
%by substituting 
$\phi(x)=\sqrt{\alpha}P(x)$. Since $\alpha P=K - \beta Q$, we have
$$
\sigma_1 \geq \left. \int\limits_{\Omega}r (K-\beta Q) \left(1-\frac{\displaystyle u_s+v_s}
{\displaystyle  K}\right)\,dx \right/ \int\limits_{\Omega}\alpha P \,dx 
$$
However, the numerator in the right-hand side equals
$$
\int\limits_{\Omega}r K \left(1-\frac{\displaystyle u_s+v_s}
{\displaystyle  K}\right)\,dx+\beta \int\limits_{\Omega} r Q \left(\frac{\displaystyle u_s+v_s}
{\displaystyle  K}-1\right)\,dx>0, 
$$
from (\ref{c5_e1}) and (\ref{c5_e2}), thus $\sigma_1>0$. However,
$u_s$ is the solution of 
\begin{equation}
\begin{array}{ll}
\nabla \cdot \left[ a_1 \nabla \left(\frac{u_s}{P} \right)\right]
+r u_s \left(1-\frac{\displaystyle u_s+v_s}{\displaystyle K}\right)
= 0, & \;x\in\Omega, \\
\frac{\displaystyle \partial (u_s/P)}{\displaystyle \partial n}=0, \;x\in\partial\Omega
\end{array}
\label{c5_eig1abc}
\end{equation}
and thus is a positive principal eigenfunction of (\ref{c5_eig1}) associated with the principal eigenvalue 0. Thus 
$(u_s,v_s)=(\alpha P(x),\beta Q(x))$ is the unique coexistence solution of (\ref{ssystem}), 
whenever $K(x)=\alpha P(x)+\beta Q(x)$ and $P$,$Q$ are linearly independent.
\end{proof}

Using Lemma~\ref{Lsteady1abc} and the same scheme as in the proof of Lemma~\ref{coex_r1}, we obtain

\begin{Lemma}\label{coex_r1abc}
Assume that $P(x)/K(x)$ is constant on $\Omega$ and $Q(x)$, $K(x)$ are linearly independent, 
then the system (\ref{ssystem}) has no coexistence equilibrium.
\end{Lemma}

Under the assumptions of Lemma~\ref{coex_r1abc}, system \eqref{ssystem} has the semi-trivial equilibrium $(K,0)$.

We recall that, according to the theory of monotone dynamical systems, once the trivial equilibrium is a 
repeller, both semi-trivial equilibrium solutions are unstable, the coexistence equilibrium is globally 
asymptotically stable.

\begin{Th}\label{c5_Th6}
Let $P(x)$ and $Q(x)$ be linearly idependent on $\Omega$, and
$K(x)\equiv \alpha P+\beta Q$, where $\alpha>0$, $\beta>0$. Then the unique coexistence solution 
$(u_s,v_s)\equiv (\alpha P(x),\beta Q(x))$ of \eqref{ssystem} is globally asymptotically stable. 
\end{Th}

Note that once the trivial equilibrium is a    
repeller, there is no coexistence equilibrium and one of the two semi-trivial equilibrium solutions is 
unstable, the other one is globally asymptotically stable.
Using  Lemmata~\ref{semi_p1abc} and \ref{coex_r1abc}, we can prove the following result.

\begin{Th}\label{c5_Th7}
Let $P(x)/K(x)$ be constant, $P$ and $Q$ be linearly independent on $\Omega$.
Then the semi-trivial equilibrium 
$(K(x),0)$ of (\ref{ssystem}) is globally asymptotically stable. 
\end{Th}

\section{Influence of Diffusion Coefficients and Intrinsic Growth Rates on Competition Outcome}
\label{sec:influence}

Further, we study the dependency of the scenario (competitive exclusion or coexistence) in the
case when $K$ is not in a positive hull of $P$ and $Q$. To this end, we assume that $a_1$ and $a_2$ are 
proportional, and $r$ is multiplied by two different constants
\begin{align}\label{ssystem1}
\begin{cases}
\frac{\displaystyle \partial u}{\displaystyle \partial t}
= \nabla \cdot \left[ \frac{\displaystyle d_1 a(x)}{\displaystyle P(x)} \left( \nabla u(t,x)- u(t,x) \frac{\displaystyle
\nabla P}{\displaystyle P} \right) \right]
+r_1 r(x)u(t,x)\left(1-\frac{\displaystyle u(t,x)+ v(t,x)}{\displaystyle K(x)}\right), \vspace{3mm}
\\
\frac{\displaystyle \partial v}{\displaystyle \partial t}
= \nabla \cdot \left[ \frac{\displaystyle d_2 a(x)}{\displaystyle Q(x)} \left( \nabla v(t,x)- v(t,x) \frac{\nabla Q}{Q}
\right) \right] + r_2 r(x)v(t,x)\left(1-\frac{\displaystyle  u(t,x)+v(t,x) }
{\displaystyle K(x)}\right),\\
t>0,\;x\in{\Omega},\\
\frac{\displaystyle \partial u}{\displaystyle \partial n} - \frac{\displaystyle u}{\displaystyle P} \frac{\displaystyle
\partial P}{\displaystyle \partial n}
=\frac{\displaystyle \partial v}{\displaystyle \partial n}  -
\frac{\displaystyle v}{\displaystyle Q} \frac{\displaystyle \partial Q}{\displaystyle \partial n} =0,~x\in 
{\partial}{\Omega},\\
u(0,x)=u_0(x), \;v(0,x)=v_0(x),\;x\in{\Omega}.
\end{cases}
\end{align}

\begin{Lemma}
\label{Lsmalld}
Let $K$, $P$ and $K$, $Q$ be linearly independent, $d_2$ and $r_2$  be fixed. Then, for a fixed $r_1$ there is $d^{\ast}$ such that
for $d_1<d^{\ast}$, the semi-trivial equilibrium $(0,v^{\ast})$ is unstable. For a fixed $d_1$, there is $r^{\ast}$
such that for $r_1>r^{\ast}$, the semi-trivial equilibrium $(0,v^{\ast})$ of \eqref{ssystem1} is unstable. 
\end{Lemma}
\begin{proof}
Consider the eigenvalue problem  associated with the first equation in \eqref{ssystem1}
around $(0,v^{\ast})$
 \begin{equation}\label{eig_p1abc}
\nabla \cdot \left[ d_1 a(x) \nabla \left(\frac{\displaystyle \phi(x) }{\displaystyle P(x)}\right) \right]
+r_1 r(x)\phi(x)  \left(1-\frac{\displaystyle v^{\ast}(x)}{\displaystyle K(x)}\right)=\sigma \phi(x),\; x\in \Omega,\;
 \frac{\displaystyle \partial (\phi/P)}{\displaystyle \partial n}=0,\; x\in\partial\Omega.
\end{equation}
The principal eigenvalue of \eqref{eig_p1abc} is defined as
$$
\sigma_1 =
\sup_{\phi \neq 0, \phi\in W^{1,2}} \left. \left[
-\int \limits_\Omega d_1 a |\nabla (\phi/P)|^2\,dx
+\int \limits_\Omega r_1 r(x)\frac{\phi^2}{P} \left(1-\frac{\displaystyle v^{\ast}}{\displaystyle K}\right)\,dx\right]
\right/
\int \limits_\Omega \frac{\phi^2}{P}\,dx,
$$
and the semi-trivial equilibrium $(0,v^{\ast})$ is unstable if we can find $\phi$ such that the expression in the right hand
side is positive. Taking $\phi=\sqrt{KP}$ and using the fact
that for linearly independent $K$, $Q$
$$
M:= \int \limits_\Omega r(x) K(x)  \left(1-\frac{\displaystyle v^{\ast}}{\displaystyle K}\right)\,dx >0 
$$
we obtain that
$$
-\int \limits_\Omega d_1 a |\nabla (\phi/P)|^2\,dx
+\int \limits_\Omega r_1 r(x)\frac{\phi^2}{P} \left(1-\frac{\displaystyle v^{\ast}}{\displaystyle 
K}\right)\,dx
=
-\int \limits_\Omega d_1 a |\nabla (\sqrt{K/P})|^2\,dx + r_1 M >0
$$
whenever either
$$d_1< d^{\ast} := r_1 M \left(\int \limits_\Omega a |\nabla (\sqrt{K/P})|^2\,dx \right)^{-1}
$$
or
$$
r_1> r^{\ast} := \frac{d_1}{M} \int \limits_\Omega a |\nabla (\sqrt{K/P})|^2\,dx,
$$
which concludes the proof. 
\end{proof}

Similarly, for fixed $r_1$ and $d_1$, for any fixed $r_2$ we can find $d^{\ast}$ and for any fixed $d_2$ there is $r^{\ast}$ such that 
for $d_2< d^{\ast}$ or $r_2> r^{\ast}$, respectively, the equilibrium $(u^{\ast},0)$ is unstable, and the second species survives.
Thus, by either slowing its dispersal speed or increasing its intrinsic growth rate, the species can provide
its survival, unless the other chooses the optimal strategy proportional to $K$.

However, from Lemma~\ref{Lsmalld} we cannot conclude that for small $d_i$ there is coexistence, 
as well as for large $r_i$,
since the semi-trivial solutions depend on these constants.

Next, let us consider the cases when both populations have the same diffusion strategy with $P(x) \equiv Q(x)$, which is 
linearly independent of $K(x)$. 

\begin{Lemma}
\label{higher_diff_coex}
Let $a_1(x)=a_2(x)=a(x)$, $P(x) \equiv Q(x)$ satisfy 
\begin{equation}
\label{noncorrespondence}
\nabla \cdot \left[ a \nabla \left( \frac{K}{P} \right) \right] \not\equiv 0, ~~x \in \Omega,
\end{equation}
$r_1=r_2=1$, $d_1<d_2$. Then
there is no coexistence equilibrium of \eqref{ssystem}. 
\end{Lemma}
\begin{proof}
We assume that there is a coexistence equilibrium $(u_s,v_s)$.
Following the proof of Lemma~\ref{coex_r1}, consider the eigenvalue 
problems
\begin{equation}
\nabla \cdot \left[ d_1 a(x) \nabla \left(\frac{\phi(x)}{P(x)} \right)\right]
+r\phi(x) \left(1-\frac{\displaystyle u_s+v_s}{\displaystyle K}\right)
=\sigma\phi(x), \;x\in\Omega,
\frac{\displaystyle \partial (\phi/P)}{\displaystyle \partial n}=0, \;x\in\partial\Omega
\label{c5_eig_abc1}
\end{equation}
and
\begin{equation}
\nabla \cdot \left[ d_2 a(x) \nabla \left(\frac{\psi(x)}{P(x)} \right)\right]
+r\psi(x) \left(1-\frac{\displaystyle u_s+v_s}{\displaystyle K}\right)
=\sigma\psi(x), \;x\in\Omega,
\frac{\displaystyle \partial (\psi/P)}{\displaystyle \partial n}=0, \;x\in\partial\Omega.
\label{c5_eig_abc2}
\end{equation}
The principal eigenvalue $\tilde{\sigma}_1$  of \eqref{c5_eig_abc2} is defined as
\begin{equation}
\tilde{\sigma}_1
= \left. \sup_{\psi \neq 0, \psi\in W^{1,2}} \left[-
\int \limits_\Omega d_2 a |\nabla (\psi/P)|^2\,dx
+\int \limits_\Omega r \frac{\psi^2}{P} \left(1-\frac{\displaystyle u_s+v_s}{\displaystyle K}\right)\,dx\right] \right/
\int \limits_\Omega \frac{\psi^2}{P}\,dx.
\label{add2abc}
\end{equation}
However, since $(u_s,v_s)$ is an equilibrium solution, the function $v_s$ satisfies
$$
\nabla \cdot \left[ d_2 a \nabla \left(\frac{v_s}{P} \right)\right]
+r v_s \left(1-\frac{\displaystyle u_s+v_s}{\displaystyle K}\right)
=0, \;x\in\Omega,
\frac{\displaystyle \partial (v_s/P)}{\displaystyle \partial n}=0, \;x\in\partial\Omega
$$
and is consequently a positive principal eigenfunction of  \eqref{c5_eig_abc2} corresponding
to the principal eigenvalue $\tilde{\sigma}_1 = 0$.
According to \eqref{add2abc}, 
\begin{equation}
- \int \limits_\Omega d_2 a |\nabla (v_s/P)|^2\,dx
 + \int \limits_\Omega r \frac{v_s^2}{P} 
\left(1-\frac{\displaystyle u_s+v_s}{\displaystyle K}\right)\,dx  =0
\label{add2ab}    
\end{equation} 
Also, as $(u_s,v_s)$ is an equilibrium solution, $u_s$ satisfies
$$
\nabla \cdot \left[ d_1 a \nabla \left(\frac{u_s}{P} \right)\right]
+r u_s \left(1-\frac{\displaystyle u_s+v_s}{\displaystyle K}\right)
=0, \;x\in\Omega,
\frac{\displaystyle \partial (u_s/P)}{\displaystyle \partial n}=0, \;x\in\partial\Omega,
$$
The principal eigenvalue $\sigma_1$  of \eqref{c5_eig_abc1} is defined as
\begin{equation}
\sigma_1
= \left. \sup_{\phi \neq 0, \phi\in W^{1,2}} \left[-
\int \limits_\Omega d_1 a |\nabla (\phi/P)|^2\,dx
+\int \limits_\Omega r \frac{\phi^2}{P} 
\left(1-\frac{\displaystyle u_s+v_s}{\displaystyle K}\right)\,dx\right] \right/
\int \limits_\Omega \frac{\phi^2}{P}\,dx.
\label{add1abc}
\end{equation}
Substituting $\phi=v_s$ we get by \eqref{add2ab}
\begin{eqnarray*}
& & -\int \limits_\Omega d_1 a |\nabla (v_s/P)|^2\,dx
+\int \limits_\Omega r \frac{v_s^2}{P}
\left(1-\frac{\displaystyle u_s+v_s}{\displaystyle K}\right)\,dx
\\
& = & (d_2-d_1)\int \limits_\Omega a |\nabla (v_s/P)|^2\,dx
\\
& & + \left[ - \int \limits_\Omega d_2 a |\nabla (v_s/P)|^2\,dx
 + \int \limits_\Omega r(x)\frac{v_s^2}{P}
\left(1-\frac{\displaystyle u_s+v_s}{\displaystyle K}\right)\,dx  \right] \\
& = &
(d_2-d_1)\int \limits_\Omega  a |\nabla (v_s/P)|^2\,dx >0, 
\end{eqnarray*}
unless $v_s/P$ is constant. However, $v_s/P\equiv \alpha$ implies $u_s+v_s \equiv K$ 
on $\Omega$. Substituting $u_s=K-\alpha P$ in the first equation of \eqref{ssystem}
yields on $\Omega$ that
$$
0=\nabla \cdot \left[ \frac{\displaystyle a}{\displaystyle P} \left( \nabla (K-\alpha P)- (K-\alpha P) 
\frac{\displaystyle \nabla P}{\displaystyle P} \right) \right]=
\nabla \cdot \left[  a \nabla \left( \frac{K}{P} \right) \right],
$$
which contradicts to \eqref{noncorrespondence} in the assumption of the lemma.
Thus $\sigma_1>0$. Let us note that $u_s$ satisfies
$$
\nabla \cdot \left[ d_1 a \nabla \left(\frac{u_s}{P} \right)\right]
+r u_s \left(1-\frac{\displaystyle u_s+v_s}{\displaystyle K}\right)
=0, \;x\in\Omega,
\frac{\displaystyle \partial (u_s/P)}{\displaystyle \partial n}=0, \;x\in\partial\Omega
$$
and thus is the positive principal eigenfunction of \eqref{c5_eig_abc1} corresponding to the 
principal eigenvalue $\sigma_1=0$.
The contradiction proves that there is no coexistence equilibrium.
\end{proof}

%The proof of the following result is similar.

\begin{Lemma}
\label{higher_diff_semi}
Let $P(x) \equiv Q(x)$ be non-proportional to $K(x)$ on $\Omega$, $r_1=r_2=1$, $d_1<d_2$. Then
the semi-trivial equilibrium $(0,v^{\ast})$ of \eqref{ssystem1} is unstable.
\end{Lemma}
\begin{proof}
The semi-trivial equilibrium $(0,v^{\ast})$ is a solution of the problem
\begin{equation}
\nabla \cdot \left[ d_2 a \nabla \left(\frac{\psi(x)}{P(x)} \right)\right]
+r\psi(x) \left(1-\frac{\displaystyle v^{\ast}}{\displaystyle K}\right)
=\sigma\psi(x), \;x\in\Omega,
\frac{\displaystyle \partial (\psi/P)}{\displaystyle \partial n}=0, 
\;x\in\partial\Omega
\label{c5_semi_abc1a}
\end{equation}
and thus $v^{\ast}$ is a positive principal eigenfunction corresponding to the zero eigenvalue 
of the problem
\begin{equation}
\nabla \cdot \left[ d_2 a \nabla \left(\frac{v^{\ast}(x)}{P(x)} \right)\right]
+r v^{\ast}(x) \left(1-\frac{\displaystyle v^{\ast}}{\displaystyle K}\right)
= 0, \;x\in\Omega,
\frac{\displaystyle \partial (v^{\ast}/P)}{\displaystyle \partial n}=0, \;x\in\partial\Omega.
\label{c5_eig_abc2a}
\end{equation}
Integrating over $\Omega$ and using the boundary conditions, we obtain 
%\cite{CC},
\begin{equation}
- \int \limits_\Omega d_2 a |\nabla (v^{\ast}/P)|^2\,dx
 + \int \limits_\Omega r(x)\frac{(v^{\ast})^2}{P}
\left(1-\frac{\displaystyle v^{\ast}}{\displaystyle K}\right)\,dx  =0.
\label{add2ab_semi}
\end{equation}
To explore local stability, we notice that the linearization at $(0,v^{\ast})$ has the form 
(see, for example, \cite{Kor2})
\begin{align*}
\begin{cases}
\frac{\displaystyle\partial u(t,x)}{\displaystyle\partial t}
=\displaystyle  \nabla \cdot \left[ d_1 a(x) \nabla \left(\frac{u(t,x)}{P(x)} \right)\right] +
r(x)u(t,x)\left(1-\frac{\displaystyle v^{*}(x)}{\displaystyle K(x)}\right), \\
%\;t> 0,\;x\in\Omega, \\
\frac{\displaystyle \partial v(t,x)}{\displaystyle\partial t}  
=\nabla \cdot \left[ d_2 a(x) \nabla \left(\frac{v(t,x)}{P(x)} \right)\right]+
r(x)v(t,x)\left(1-\frac{\displaystyle 2v^{*}(x)}{\displaystyle K(x)}\right)-r(x)v^{*}(x)\frac{u(t,x)}{K(x)}, \\
%\;
t> 0,\; \;x\in\Omega, \\
\frac{\displaystyle\partial (u/P)}{\displaystyle\partial n}
=\frac{\displaystyle\partial (v/P)}{\displaystyle\partial n}=0,\;x\in\partial\Omega
\end{cases}
\end{align*}
and study the associated eigenvalue problems
\begin{align}
\label{eig1}
\hspace{-3mm}
& \nabla \cdot \left[ d_1 a(x) \nabla \left(\frac{\phi(x)}{P(x)} \right)\right]+
r(x)\phi(x)\left(1-\frac{\displaystyle v^{*}(x)}
{\displaystyle K(x)}\right)=\sigma\phi(x), \;x\in\Omega\nonumber,\\
&
\frac{\partial(\phi/P)}{\partial n}=0, \;x\in\partial\Omega,\\
&
 \nabla \cdot \left[ d_2 a(x) \nabla \left(\frac{\psi(x)}{P(x)} \right)\right] +
r(x)\psi(x)\left(1-\frac{\displaystyle 2v^{*}(x)}{\displaystyle 
K(x)}\right)-r(x)v^{*}(x)\frac{\phi(x)}{K(x)}=\sigma\psi(x), \;x\in\Omega, \nonumber\\
&
\frac{\partial(\psi/P)}{\partial n}=0, \;x\in\partial\Omega.
\end{align}

%Similarly to the proof of Lemma~\ref{higher_diff_coex}, consider the eigenvalue 
%problems at the semi-trivial equilibrium $(0,v^{\ast})$ 
%\begin{equation}
%\nabla \cdot \left[ d_1 a(x) \nabla \left(\frac{\phi(x)}{P(x)} \right)\right]
%+r\psi(x) \left(1-\frac{\displaystyle v^{\ast}}{\displaystyle K}\right)
%=\sigma\phi(x), \;x\in\Omega,
%\frac{\displaystyle \partial (\phi/P)}{\displaystyle \partial n}=0, \;x\in\partial\Omega
%\label{c5_eig_abc1a}
%\end{equation}
%and
%\begin{equation}
%\nabla \cdot \left[ d_2 a(x) \nabla \left(\frac{\psi(x)}{P(x)} \right)\right]
%+r\phi(x) \left(1-\frac{\displaystyle v^{\ast}}{\displaystyle K}\right)
%=\sigma\psi(x), \;x\in\Omega,
%\frac{\displaystyle \partial (\psi/P)}{\displaystyle \partial n}=0, \;x\in\partial\Omega
%\label{c5_eig_linearize}
%\end{equation}

The principal eigenvalue $\sigma_1$  of \eqref{eig1} satisfies \cite{CC}
$$
\sigma_1
= \left. \sup_{\phi \neq 0, \phi\in W^{1,2}} \left[-
\int \limits_\Omega d_1 a |\nabla (\phi/P)|^2\,dx
+\int \limits_\Omega r \frac{\phi^2}{P} \left(1-\frac{\displaystyle v^{\ast}}{\displaystyle K}\right)\,dx\right] \right/
\int \limits_\Omega \frac{\phi^2}{P}\,dx.
$$
Substituting $\phi=v^{\ast}$, we obtain using \eqref{add2ab_semi}
$$
\sigma_1 \geq \left. \left[-
\int \limits_\Omega d_1 a |\nabla (v^{\ast}/P)|^2\,dx
+\int \limits_\Omega r(x)\frac{(v^{\ast})^2}{P} \left(1-\frac{\displaystyle v^{\ast}}{\displaystyle K}\right)\,dx\right] \right/
\int \limits_\Omega \frac{(v^{\ast})^2}{P}\,dx.
$$
However,
\begin{eqnarray*}
& & - \int \limits_\Omega d_1 a |\nabla (v^{\ast}/P)|^2\,dx
+\int \limits_\Omega r \frac{(v^{\ast})^2}{P} \left(1-\frac{\displaystyle v^{\ast}}{\displaystyle K}\right)\,dx
\\
&= &(d_2-d_1) \int \limits_\Omega a |\nabla (v^{\ast}/P)|^2\,dx +
\left[-
\int \limits_\Omega d_2 a |\nabla (v^{\ast}/P)|^2\,dx
+\int \limits_\Omega r \frac{(v^{\ast})^2}{P} \left(1-\frac{\displaystyle v^{\ast}}{\displaystyle K}\right)\,dx\right]
\\
& = & (d_2-d_1) \int \limits_\Omega  a |\nabla (v^{\ast}/P)|^2\,dx + 0 > 0,
\end{eqnarray*}
as $v^{\ast}/P$ is non-constant. In fact, assuming constant $v^{\ast}/P \equiv \alpha$, 
we obtain from \eqref{c5_eig_abc2a} that 
$\displaystyle r v^{\ast}(x) \left(1-\frac{\displaystyle v^{\ast}}{\displaystyle K}\right)
= 0$ for any $x\in\Omega$, or $v^{\ast} \equiv K  \equiv \alpha P$,
which contradicts to the assumption of the lemma that $P$ is not proportional to $K$ on $\Omega$.

Therefore the principal eigenvalue of the linearized problem is positive, which implies that the equilibrium
$(0,v^{\ast})$ is unstable and concludes the proof.
\end{proof}

\begin{Rem}
The assumption that $P$ is not proportional to $K$ on $\Omega$ is a particular case of assumption
\eqref{noncorrespondence}. If we assume a constant $a(x)$ in $\Omega$, condition
\eqref{noncorrespondence} means that the function $K(x)/P(x)$ is not a harmonic function on $\Omega$,
compared to being non-constant.
\end{Rem}

\begin{Th}\label{Th_diff}
Let $P(x) \equiv Q(x)$ be non-proportional to $K(x)$ on $\Omega$, $r_1=r_2=1$, $d_1<d_2$. Then
the semi-trivial equilibrium
$(u^{\ast}(x),0)$ of (\ref{ssystem1}) is globally asymptotically stable.
\end{Th}

\begin{Rem}
Theorem~\ref{Th_diff} generalizes the results of \cite{Dock} to a more general type of diffusion
in the case of two species.
Let us also note that, in the absence of diffusion, the solution of each single-species equation in
\eqref{ssystem1} is $K$. The higher the diffusion is, the more the stationary solution deviates from $K$.
\end{Rem} 

Similarly, the following result is obtained.

\begin{Th}\label{Th_growth}
Let $P(x) \equiv Q(x)$ be non-proportional to $K(x)$ on $\Omega$, $r_1>r_2$, $d_1=d_2=1$. Then
the semi-trivial equilibrium
$(u^{\ast}(x),0)$ of (\ref{ssystem1}) is globally asymptotically stable.
\end{Th}

%{\bf The results of this and the previous theorem can be verified numerically. It would be interesting to see that we
%do have coexistence when K=P=Q depending on the initial conditions.}

\section{Discussion}
\label{sec:discussion}

Our attempt to find the type of dispersal which includes previously known models as special cases, initiates the following question:
what is the flexibility of strategies that can lead to the ideal free distribution as a stationary solution?
For example, in \cite{Brav1,Brav,Kor1,Kor2,Kor4,Kor3} in the term $\Delta(u/P)$ the dispersal strategy $P$ was usually chosen 
as $P \equiv K$ guaranteeing that $K$ is a (globally stable) positive solution of the equation.
However, if $P(x) \equiv K(x)/h(x)$, where $h$ is any harmonic function on $\Omega$, $K$ is still a solution. The boundary 
conditions give that the normal derivative of $h$ on the boundary vanishes which by the Maximum Principle reduces 
all acceptable strategies to $P(x) \equiv \alpha K(x)$, $\alpha>0$. However, the constant $\alpha$ can be taken as a part of the 
diffusion coefficient, so the strategy $P \equiv K$ is to some extent a unique optimal strategy.

The results of Section~\ref{sec:influence} extend the findings of \cite{Dock} to a more general type of diffusion
and outline the coupling of the two parameters involved in the system: diffusion coefficients and intrinsic growth rates.
The system of two equations, with the same intrinsic growth rates and $\alpha$ times smaller diffusion coefficient in the first 
equation has the same stationary semi-trivial or coexistence solution as the system with the same diffusion coefficient
and $\alpha$ times larger intrinsic growth rate in the first equation.
Combining this idea with the eigenvalue technique developed in \cite{CC}
allows to look at the results of \cite{Dock} from a different perspective: the most productive type survives with the same rate 
of dispersal, all other parameters being the same, which is certainly biologically feasible.

Some of previously obtained results \cite{Brav1} can readily be extended 
to the model generalizing \eqref{ssystem} to the case of two different carrying capacities
\begin{align*}
\begin{cases}
\frac{\displaystyle \partial u}{\displaystyle \partial t}  
= \nabla \cdot \left[ \frac{\displaystyle a_1(x)}{\displaystyle P(x)} \left( \nabla u(t,x)- u(t,x) \frac{\displaystyle 
\nabla P}{\displaystyle P} \right) \right]
+r_1(x)u(t,x)\left(1-\frac{\displaystyle u(t,x)+ v(t,x)}{\displaystyle K_1(x)}\right), \vspace{3mm}
\\
\frac{\displaystyle \partial v}{\displaystyle \partial t} 
= \nabla \cdot \left[ \frac{\displaystyle a_2(x)}{\displaystyle Q(x)} \left( \nabla v(t,x)- v(t,x) \frac{\nabla Q}{Q} 
\right) \right] + r_2(x)v(t,x)\left(1-\frac{\displaystyle  u(t,x)+v(t,x) }
{\displaystyle K_2(x)}\right),\\
t>0,\;x\in{\Omega},\\
\frac{\displaystyle \partial u}{\displaystyle \partial n} - \frac{\displaystyle u}{\displaystyle P} \frac{\displaystyle 
\partial P}{\displaystyle \partial n}  
=\frac{\displaystyle \partial v}{\displaystyle \partial n}  - 
\frac{\displaystyle v}{\displaystyle Q} \frac{\displaystyle \partial Q}{\displaystyle \partial n} =0,~x\in 
{\partial}{\Omega},\\
u(0,x)=u_0(x), \;v(0,x)=v_0(x),\;x\in{\Omega},
\end{cases}
\end{align*}
at least in the part of the competitive exclusion of $v$ when, for example, $P \equiv \alpha K_1$, $r_1/r_2$ is constant 
and $K_1 \geq K_2$ on $\Omega$.

The results obtained in the present paper strongly rely on the theory of monotone dynamical systems
\cite{Hs}, and are not applicable to the case of more than two competing species. The case of three species
for a particular diffusion strategy was considered in \cite{LouMunther}. It would be interesting to 
incorporate the approach of the present paper to dispersal with consideration of more than two (and probably three)
competing species, as well as patchy environment \cite{C5}. 

%!!!!!!!!!!!!!!!!!!!!!!!!!!!

In spite of the variety of diffusion models described in the present paper, they are based on the same hypotheses:
\begin{enumerate}
\item
The domain is in some sense isolated, there is no flux across the boundary, which corresponds to 
the Neumann boundary conditions. This describes a closed ecosystem, all the change is subject to spatially-dependent
growth laws.
\item
The initial-boundary value problem is designed in such a way that any nontrivial initial conditions lead to a positive
and bounded solution.
\item
Each population diffuses on its own, i.e. the corresponding diffusion terms do not include the second species. 
\end{enumerate}

The third hypothesis was challenged in \cite{Shigesada}, where it was suggested to consider the diffusion of $u$
of the form $\Delta(u(d_1+f_1(v)))$, with $\Delta(v(d_2+f_2(u)))$ for the 
second species. If $f_i \equiv 0$ and in our paper
we assume constant $P$, $Q$, these two models coincide; otherwise, they are independent. 
The approach of \cite{Shigesada} was further developed in \cite{Jia}, with the homogeneous Dirichlet boundary conditions.
As justified in \cite{Jia}, a solution is not necessarily positive, generally, the origin is a repeller: with the appropriate
design of cross-diffusion, either $u$ or $v$, or both, can move out of the considered domain. 
Compared to \cite[Propositions 2.3]{Jia}, with conditions including an unknown semi-trivial solution,
Theorem~\ref{c5_Th7} includes an explicit stability test; moreover, while \cite[Propositions 2.3]{Jia} deals with local asymptotical stability,
in the present paper we analyze the global behaviour.
This can be illustrated with  \cite[Propositions 4.3]{Jia}: convergence to a coexistence equilibrium is stipulated by
certain initial conditions, compared to the unconditional result of Theorem~\ref{c5_Th6}.
One of the differences is that \cite[Propositions 4.3]{Jia} assumes non-symmetric interaction (growth) type, compared to the 
symmetric case in the present paper. 
The results of  \cite[Propositions 4.3]{Jia} can be compared to \cite{Brav1}, where different carrying capacities 
for $u$ and $v$ were considered, and also coexistence was observed. 

However, the results of \cite{Jia}, where different diffusion type and boundary conditions were investigated,
outline the same relations as the conclusions of Section~\ref{sec:influence}: the detrimental role of self-diffusion rates and
higher growth rates being a positive factor. 

The conclusions on cross-diffusion (see \cite[Remark 4.3]{Jia}) are not applicable 
to the model considered in the present paper. It would be interesting to combine the ideas of \cite{Shigesada} on cross-diffusion
with the above hypotheses and the general type of self-diffusion, 
exploring cross-interactions not only in the growth but also in the 
diffusion part.

%!!!!!!!!!!!!!!!!!!!!!!!!!!!

\section{Acknowledgment}

The authors are grateful to anonymous reviewers for their valuable comments that significantly contributed to the presentation of the paper,
to L. Korobenko for her constructive suggestions on the previous 
versions of the manuscript, to Prof. Cosner for fruitful discussions during ICMA-V meeting in 2015. 
The research was supported by NSERC grant RGPIN-2015-05976.


\begin{thebibliography}{99}

\bibitem{Averill}
I. Averill, Y. Lou and D. Munther,
On several conjectures from evolution of dispersal, {\em J. Biol. Dyn.} {\bf 6} (2012), 
%No. 2, 
117--130.

%\bibitem{Lou_Ni}
%Y.\, Lou and W.\, M. Ni,
%Diffusion, self-diffusion and cross-diffusion, {\em J. Differential Equations} {\bf 131} (1996), 79--131.

%\bibitem{Brown}
%P.\, N. Brown,
%Decay to uniform states in ecological interactions, {\em SIAM J. Appl. Math.} {\bf 38} (1980), No. 1, 22--37.

\bibitem{Brav1}
E. Braverman, Md. Kamrujjaman and L. Korobenko,
Competitive spatially distributed population dynamics models: does diversity in diffusion
strategies promote coexistence? {\em Math. Biosci.} {\bf 264} (2015), 63-–73.


\bibitem{Brav}
E. Braverman and L. Braverman,
Optimal harvesting of diffusive models in a non-homogeneous
environment, {\em Nonlin. Anal. Theory Meth. Appl.}
{\bf 71} (2009), e2173--e2181.

%\bibitem{MT}
%Y. Morita and K. Tachibana, An entire solution to the Lotka-Volterra
%competition-diffusion equations, {\em SIAM J. Math. Anal.}, {\bf 39} (2009), No. 6, 2217--2240.

%\bibitem{PDM}
%Piero de Mottoni, 
%Qualitative analysis for some quasilinear parabolic systems, {\em Inst. Math. Polish Acad. Sci.},
%1979.

%\bibitem{NWM}
%Ni and W. Ming, The Mathematics of Diffusion, {\em Society for Industrial and Applied Mathematics},
%2011.

\bibitem{CC}
R.\,S. Cantrell and C. Cosner, Spatial Ecology via Reaction-diffusion 
Equations, {\em Wiley Series in 
Mathematical and Computational Biology}, John Wiley \& Sons, 
Chichester, 2003.

%\bibitem{C0}
%R.\,S. Cantrell, C. Cosner, D.\,L. Deangelis, V. Padron,
%The ideal free distribution as an evolutionarily stable
%strategy, {\em J. Biol. Dyn.} {\bf 1} (2007),
%no. 3,
%249-–271.

\bibitem{C4}
R.\,S. Cantrell, C. Cosner and Y. Lou, 
Movement toward better environments
and the evolution of rapid diffusion,
{\em Math. Biosci.} {\bf 204} (2006), 199--214.


%\bibitem{C1}
%R.\,S. Cantrell, C. Cosner, Y. Lou, Advection-mediated coexistence of competing species, {\em
%Proc. Roy. Soc. Edinburgh Sect. A} {\bf 137} (2007),
%no. 3,
%497--518.

\bibitem{C2}
R.\,S. Cantrell, C. Cosner and Y. Lou, 
Approximating the ideal free distribution via reaction-diffusion-advection
equations, {\em J. Differential Equations} {\bf 245} (2008),
%no. 12,
3687-–3703.

\bibitem{C3}
R.\,S. Cantrell, C. Cosner and Y. Lou, 
Evolution of dispersal and the ideal free
distribution, {\em Math. Biosci. Eng.} {\bf 7} (2010),
%no. 1,
17--36.

\bibitem{C5}
R.\,S.  Cantrell, C. Cosner and Y. Lou, 
Evolutionary stability of ideal free dispersal strategies in patchy environments,
{\em J. Math. Biol.} {\bf 65} (2012),
%no. 5,
943-–965.

%\bibitem{Chen1}
%X. Chen, K.Y. Lam, Y. Lou, Dynamics of a reaction-diffusion-advection model for two competing species,
%{\em Discrete Contin. Dyn. Syst.} {\bf 32} (2012), %no. 11,
%3841-–3859.

%\bibitem{DL_JMB_1996}
%U. Dieckmann, R. Law,
%The dynamical theory of coevolution: a derivation from stochastic ecological processes,
%{\em J. Math. Biol.} {\bf 34} (1996), 579-–612.

\bibitem{Dock} 
J. Dockery, V. Hutson, K. Mischaikow and M. Pernarowski, 
The evolution of slow dispersal rates: a reaction diffusion model, 
{\em J. Math. Biol.} {\bf 37} 
%(1) 
(1998), 61--83.

%\bibitem{Geritz}
%S.\,A.\,H. Geritz, M. Gyllenberg,
%The Mathematical Theory of Adaptive Dynamics, Cambridge University Press, Cambridge, 2008.

%\bibitem{William} 
%S.\,Williams and P.\,Chow, Nonlinear reaction-diffusion models for interacting populations, 
%{\em J. Math. Anal. Appl.}, {\bf 62}  (1978), 157--159.

\bibitem{Kor1} 
L. Korobenko and E. Braverman, 
A logistic model with a carrying capacity driven diffusion, 
{\em Can. Appl. Math. Quart.} {\bf 17} (2009), 85--100.

\bibitem{Kor2} 
L. Korobenko and E. Braverman, On logistic models with a carrying capacity dependent diffusion: 
stability of equilibria and coexistence with a regularly diffusing population, 
{\em Nonlinear Anal. B: Real World Appl.} {\bf 13} (2012), 
%No. 6, 
2648--2658.

\bibitem{Kor4}
L. Korobenko and E. Braverman, 
On evolutionary stability of carrying capacity driven dispersal in competition with regularly diffusing populations,
{\em J. Math. Biol.} {\bf 69} (2014), 
%No. 5, 
1181--1206.

\bibitem{Kor3}
L. Korobenko, Md. Kamrujjaman and E. Braverman,
Persistence and extinction in spatial models with a carrying capacity driven diffusion and harvesting,
{\em J. Math. Anal. Appl.} 399 (2013), 352--368.

\bibitem{LouMunther}
Y. Lou and D. Munther,
Dynamics of a three species conpetition model, {\em  Discrete Contin. Dyn. Syst.} {\bf 32} 
%Number 9, September 2012                            
(2012), 3099--3131.


%\bibitem{Leung} 
%A.\, Leung, Limiting behaviour for a prey-predator model with diffusion and crowding effects,
%{\em J. Math. Biol.}, {\bf 6}  (1978), 87--93.

%\bibitem{Murray} 
%J.\, Murray, Non-existence of wave solutions for the class of reaction-diffusion equations given by the Volterra
%interacting-population equations with diffusion,
%{\em J. Theor. Biol.}, {\bf 52}  (1975), 459--469.

\bibitem{He}
X.\,Q. He and W.\,M. Ni, 
The effects of diffusion and spatial variation in Lotka-Volterra competition-diffusion system I: 
Heterogeneity vs. homogeneity, {\em J. Differential Equations} {\bf 254} (2013), 
%no. 2, 
528-–546. 


%\bibitem{Conway} 
%E.\, Conway and J.\, Smoller, Diffusion and the predator-prey interaction,
%{\em SIAM J. Appl. Math.}, {\bf 33}  (1977), 673--686.

%\bibitem{KW} 
%K.\, Kishimoto and H.\, F. Weinberger, The spatial homogeneity of stable equilibria of some reaction-diffusion
% systems on convex domains, {\em J. Differential Equations}, {\bf 58}  (1985), 15--21.


\bibitem{Gilb} 
D. Gilbarg and N.\,S. Trudinger, 
Elliptic Partial Differential Equations of Second Order, second edition, Springer-Verlag, Berlin, 1983.

%\bibitem{Gil} 
%M.\,E. Gilpin and F.\,J. Ayala, Global models of growth and competition, 
%{\em Proc. Natl. Acad. Sci. USA}, {\bf 70}  (1973), 3590--3593.

%\bibitem{Gompertz}
%B. Gompertz, On the nature of the function expressive of human mortality and on a a new mode of 
%determining the value of life contingencies,
%{\em Philos. Trans. Roy. Soc. London} {\bf 115} (1825), 513--583.

%\bibitem{Gosso}
%A. Gosso, V. La Morgia, P. Marchisio, O. Telve and E. Venturino,
%Does a larger carrying capacity for an exotic species allow environment invasion? - Some considerations 
%on the competition of red and grey squirrels,
%{\em J. Biol. Syst.} {\bf 20} (2012), No. 3, 221--234 
 
%\bibitem {Bly}
%W.  Gurney, S. Blythe, R. Nisbet, Nicholson's blowflies revisited, {\em Nature} {\bf 287} (1980), 17--21.


%\bibitem{Ni}
%A. Nicholson, An outline of the dynamics of animal populations,
%{\em Austral. J. Zool.} {\bf  2} (1954), 9--65.

%\bibitem{Z} L. Zhou, Y. Fu, Existence and stability 
%of periodic quasisolutions in nonlinear parabolic systems with discrete 
%delays, {\em J. Math. Anal. Appl.} {\bf 250} (2000), 139-161.

%\bibitem{B} L. Korobenko and E. Braverman, A logistic model with a carrying capacity driven
%diffusion, {\em Can. Appl. Math. Quart.} {\bf 17} (2009), 85-100.


%\bibitem{Kot} M. Kot, Elements of Mathematical Ecology, Cambridge University Press, Cambridge, 2001. 

%\bibitem{MEF} 
%M.\, Mimura, S.\, I. Ei and Q.\, Fang, Effect of domain-shape on the coexistence problems in a competition-diffusion
%system, {\em J. Math. Biol.}, {\bf 29}  (1991), 219--237.


%\bibitem{Ph} P. Hess, Periodic-Parabolic Boundary Value Problems and Positivity, Pitman Research Notes Mathematics Series
%247, Longman Scientific and Tech., John Wiley and Sons, New York, 1991. 

%\bibitem{Lam}
%K.\,Y. Lam and W.\,M. Ni,
%Uniqueness and complete dynamics in heterogeneous competition-diffusion systems,
%{\em SIAM J. Appl. Math.} {\bf 72} (2012), 1695--1712.

\bibitem{Jia}
Y. Jia, J. Wu and H.\,K. Xu, 
Positive solutions of a Lotka-Volterra competition model with cross-diffusion, 
{\em Comput. Math. Appl.} {\bf 68} (2014), 
%no. 10, 
1220-–1228. 

\bibitem{Pao} 
C. \,V. Pao, Nonlinear Parabolic and Elliptic Equations, 
Plenum, New York, 1992.

%\bibitem{Ninomiya} 
%H.\, Ninomiya, Separatrices  of competition-diffusion equations,
% {\em J. Math. Kyoto. Univ.}, {\bf 35}  (1995), 539--567.

\bibitem{Shigesada}
N. Shigesada, K. Kawasaki and E. Teramoto, 
Spatial segregation of interacting species. {\em J. Theoret. Biol.} {\bf 79} (1979), 
%no. 1, 
83-–99.


\bibitem{Hs} 
H.\,L. Smith, Monotone Dynamical Systems, An Introduction to the Theory of Competitive and Cooperative Systems,
Amer. Math. Soc., 41, Providence, RI, 1995.


%\bibitem{Prot} 
%M.\,H. Protter, H. F. Weinberger, Maximum Principles in 
%Differential Equations, Prentice-Hall, Inc., Englewood Cliffs, N.J., 1967.

%\bibitem{Schreiber_2012}
%S.\,J. Schreiber,
%The evolution of patch selection in stochastic environments,
%American Naturalist {\bf 180} (2012), %No. 1, 
%17--34.

%\bibitem{Hsu}
%S.B. Hsu, H.L. Smith and P. Waltman, 
%Competitive exclusion and coexistence for competitive systems on ordered Banach spaces, 
%{\em Trans. Amer. Math. Soc.}, {\bf 348} (1996), no. 10, 4083--4094.

%\bibitem{Hm}
%H.\, Matano, Existence of nontrivial unstable sets for equilibriums of strongly order-preserving systems, {\em J. 
%Fac. Sci. Univ. Tokyo} {\bf 30} (1984), 645-–673.

%\bibitem{Sorace}
%R. Sorace and N.\,L. Komarova, Accumulation of neutral mutations in growing cell colonies with competition, {\em J. 
%Theoret. Biol.} {\bf 314} (2012), 84-–94.

%\bibitem{Zhang}
%S. \,Zhang, L.\, Zhou and Z.\, Liu, The spatial behavior of a competition--diffusion--advection system with strong competition,
%{\em Nonlinear Anal. Real World Appl.} {\bf 14} (2013),
%no. 2,
%976-–989.

%\bibitem{End}
%E.\,N. Dancer, Positivity of maps and applications, in Topological nonlinear analysis, M. Matzeu and A. Vignoli (Eds), 
%Birkhauser, Boston, {\em Prog. Nonlinear Diff. Eqs Appl.} {\bf 15} (1995), 303-–340.

%\bibitem{Smith} 
%F. E. Smith, Population dynamics in Daphnia Magna and a new model for population, 
%{\em Ecology} {\bf 44} (1963), No. 4, 651--663.

\end{thebibliography}
\end{document}